\documentclass[a4paper,11pt,onecolumn]{amsart}
\usepackage{amsfonts}

\usepackage{amscd}
\usepackage{amsmath}
\usepackage{amssymb}
\usepackage{amsthm}
\usepackage{epsf}
\usepackage{latexsym}
\usepackage{verbatim}
\usepackage{color}
\input epsf.tex
\usepackage{bbm}

\usepackage{epstopdf}
\usepackage[pdftex]{graphicx}

\newtheorem{theorem}{Theorem}[section]

\newtheorem{proposition}[theorem]{Proposition}
\theoremstyle{definition}

\begin{document} 

\title{Shapes of Polynomial Julia Sets, Revisited }
\author{Kathryn A. Lindsey}
\address{Department of Mathematics\\
         University of Chicago\\
         Chicago, IL 60637 \\ 
         U.S.A.}        
\email{klindsey@math.uchicago.edu}

\begin{abstract}
Any finite union of disjoint, mutually exterior Jordan curves in the complex plane can be approximated arbitrarily well in the Hausdorff topology by polynomial Julia sets.  Furthermore, the proof is constructive.  
\end{abstract}

\maketitle

\section{Introduction}

It was shown in \cite{LindseyShapes} that any finite union of disjoint mutually exterior Jordan curves can be approximated arbitrarily well in the Hausdorff topology by Julia sets of \emph{rational maps} (and any single Jordan curve can be approximated by polynomial Julia sets).  Here, we give a constructive proof that any such finite union of Jordan curves can be approximated by Julia sets of \emph{polynomials}.

\begin{theorem} \label{t:main}
Let $E$ be any finite union of disjoint Jordan domains in $\mathbb{C}$.  For any $\epsilon >0$, there exists a polynomial $P$ such that $$d\left(E, \mathcal{K}(P)\right) < \epsilon, \quad
 d\left(\partial E, \mathcal{J}(P)\right)< \epsilon, \quad \textrm{and} \quad
d(\hat{\mathbb{C}} \setminus E, \hat{\mathbb{C} } \setminus \mathcal{K}(P))< \epsilon.$$
\end{theorem} 

\noindent Here, $\mathcal{K}(P)$ denotes the filled Julia set for $P$ (i.e. $\mathcal{K}(P) = \{z \in \mathbb{C} : P^m(z) \nrightarrow \infty \textrm{ as } m \rightarrow \infty\}$), $\mathcal{J}(P) = \partial \mathcal{K}(P)$ is the Julia set for $P$, and $d$ is the Hausdorff metric. 

\medskip

Approximation by Julia sets or other dynamically significant sets has been the focus of several recent works (e.g. \cite{BishopTrees, BishopPilgrim, CalegariDipoles, Ivrii, Ivrii2, LindseyShapes}).  A feature of the construction presented in this article and \cite{LindseyShapes} that is absent from \cite{BishopPilgrim, CalegariDipoles, Ivrii, Ivrii2} is that the \emph{filled Julia set}, as well as the Julia set, can be made to approximate some desired shape.  Closeness in the Hausdorff metric of a Julia set to a collection of Jordan curves does not imply closeness of the associated filled Julia set to the region bounded by those curves.  
 Approximating shapes by filled Julia sets has applications to computer graphics (\cite{Kim}).

\subsection*{Acknowledgments}
Kathryn Lindsey thanks Amie Wilkinson for her guidance during the development of this project.  
She also thanks Greg Lawler for several helpful conversations about harmonic measure.  
The author was supported by a NSF Mathematical Sciences Postdoctoral Research Fellowship.

\section{Polynomials approximating logarithmic potential} \label{s:polydef}

 Without loss of generality, we may assume that the Jordan domains Theorem \ref{t:main} are bounded by smooth Jordan curves.  Let $E$ be the union of finitely many disjoint Jordan domains in $\mathbb{C}$ bounded by smooth Jordan curves.  Write $\Omega = \hat{\mathbb{C}} \setminus E$ and $\Gamma = \partial E$.  

We will define a family of polynomials $S_{n,\delta}$ indexed by $n \in \mathbb{N}$ and $\delta > 0$.   The polynomial $S_{n,\delta}$ will have $n$ roots in $\Gamma$, distributed according to harmonic measure relative to $\infty$, and $|S_{n,\delta}|$ is an approximation of  the logarithmic potential of this measure, scaled by a real constant governed by $\delta$. 

 For background on harmonic measure and potential theory, see, for example, \cite{GarnettMarshall}.  Harmonic measure relative to $\infty$, which we will denote by $\mu_{\infty}$, is a Borel probability measure supported on $\Gamma$ that is (recall we assume $\Gamma$ consists of smooth curves) absolutely continuous with respect to arclength.  The density of harmonic measure relative to a point $w \in \Omega$ with respect to arclength at a point $\zeta \in \Gamma$ is given by the Poisson kernel, $P_w(\zeta) = \frac{-1}{2\pi} \frac{\partial g(w,\zeta)}{\partial n_{\zeta}}$, where $g(w,\cdot)$ is the Green's function with a pole at $w$ and $n_{\zeta}$ is the unit outer normal to $\Gamma$ at $\zeta$.  For $w \in \Omega$, the Green's function $g(w,\cdot):\hat{\mathbb{C}} \rightarrow \mathbb{R}$ is the unique function such that $g(w,\cdot) = 0$ on $E \cup \Gamma$, $g(w,\cdot) > 0$ on the interior of $\Omega$, $\zeta \mapsto g(w,\zeta)$ is harmonic on $\Omega \setminus \{w\}$, and $\zeta \mapsto g(w,\zeta) - \log \frac{1}{|w - \zeta|}$ is harmonic at $w$.  Harmonic measure for more general sets may be defined using probabilistic techniques, viewing harmonic measure as the hitting measure of Brownian motion, as in, for example, \cite{Lawlerbook}.  

The \emph{logarithmic potential} of the harmonic measure $\mu_{\infty}$ is the function
$$U_{\mu_{\infty}}(z) = \int \log \frac{1}{| \zeta - z|}d\mu_{\infty}(\zeta).$$
In our situation (i.e. $\Gamma$ consists of finitely many smooth Jordan curves and $\Omega$ is connected), $U_{\mu_{\infty}}$ is absolutely convergent and continuous on $\mathbb{C}$ and

\begin{equation} \label{eq:potentialbehavior}
U_{\mu_{\infty}}(z) = \gamma - g(\infty,z) \quad \textrm{for all} \quad z \in \mathbb{C},
 \end{equation}
\noindent where $\gamma$ is a real number associated to $E$ (\emph{Robin's constant} for $E$).  

\begin{theorem} \label{t:polys}
Let $E$ be a finite union of smooth Jordan domains in $\mathbb{C}$ and fix $\delta > 0$.  For each $n \in \mathbb{N}$, define 
 \begin{equation}
 S_{n, \delta}(z)  = e^{n(\gamma - \delta)} \prod_{i=1}^n (\zeta_i^n - z),
 \end{equation}
where $\{\zeta_1^n,\dots,\zeta_n^n\}$ is any set of $n$ points in $\partial E$ that are approximately equidistributed with respect to $\mu_{\infty}$.  
Then $$\lim_{n \rightarrow \infty} |S_{n,\delta}(z)| = 0 \hspace{1cm} \textrm{for all }z \textrm{ such that }g(\infty,z) < \delta, $$
$$\lim_{n \rightarrow \infty} |S_{n,\delta}(z)| = \infty \hspace{1cm} \textrm{for all }z \textrm{ such that }g(\infty,z) > \delta. $$
\end{theorem}

\noindent We must specify the meaning of ``approximately equidistributed" in Theorem \ref{t:polys}.  If $\Gamma$ consists of more than one Jordan curve, it will typically be impossible to divide $\Gamma$ into $n$ connected curves which each are of precisely $\mu_\infty$-measure $1/n$.  However, by taking $n$ to be sufficiently large, we can choose the set $\{\zeta_1^n,\dots,\zeta_n^n\}$ so that it is arbitrarily close to being equidistributed along $\Gamma$ with respect to $\mu$.  Thus, we will say that $\{\zeta_1^n,\dots,\zeta_n^n\}$ is \emph{approximately equidistributed} with respect to $\mu_\infty$ if for any $\epsilon > 0$, there exists $N \in \mathbb{N}$ such that for $n\geq N$, cutting $\Gamma$ at the points $\{\zeta_1^n,\dots,\zeta_n^n\}$ yields $n$ connected curves each of $\mu$-measure between $\frac{1}{n}-\epsilon$ and $\frac{1}{n} + \epsilon$.  

\begin{proof}
Fix $\delta>0$.  We have $$-U_{\mu_\infty}(z)  = \int_{\partial \Omega} \log |\zeta - z| d\mu_\infty(\zeta).$$ For $z \not \in \Gamma$, $z \not = \infty$, the function $\log |\zeta - z |$ is finite and continuous on $\Gamma \owns \zeta$.  Hence  $\int_{\partial \Omega} \log |\zeta - z| d\mu_\infty(\zeta)$ may be computed as a  Riemann integral with density: 

\begin{equation} \label{eq:initial}
 -U_{\mu_\infty} (z) = \lim_{n \rightarrow \infty} \frac{1}{n} \sum_{i=1}^n \log | \zeta_i^n - z|,
\end{equation} 
assuming the $\{\zeta_1^n,\dots,\zeta_n^n\}$ are equidistributed with respect to $\mu_{\infty}$.  Substituting $U_{\mu_\infty}(z) = \gamma - g(\infty,z)$ from equation (\ref{eq:potentialbehavior}) in equation (\ref{eq:initial}) gives
\begin{equation}
\lim_{n \rightarrow \infty} \frac{1}{n} \sum_{i=1}^n \log |\zeta_i^n - z| + \gamma = g(\infty,z).
\end{equation}
Thus, for any $\epsilon>0$, there exists $N \in \mathbb{N}$ such that $n \geq N$ implies 
$$ g(z,\infty) - \epsilon \leq \frac{1}{n} \sum_{i = 1}^n \log |\zeta_i^n - z| + \gamma \leq g(z,\infty) + \epsilon,$$
or, equivalently, 
$$n \left(g(\infty,z)- \epsilon \right) \leq \log \left| \prod_{i=1}^n (\zeta_i^n-z) \right| + n \gamma \leq n \left(g(z,\infty) + \epsilon \right),$$ and hence 

\begin{equation} \label{eq:bounds} 
\exp \left( n ( g(z,\infty) - \epsilon - \delta) \right) \leq  
 \left| S_{n,\delta}(z) \right| \leq \exp \left(n(g(z,\infty) + \epsilon - \delta)\right) 
\end{equation}
for any $\delta >0$.

If $g(z,\infty) + \epsilon - \delta < 0$, then $\lim_{n \rightarrow \infty} |S_{n,\delta}(z) | =0$ 
by the right hand side of inequality (\ref{eq:bounds}).  Since $\epsilon>0$ was arbitrary, we have $\lim_{n \rightarrow \infty} |S_{n,\delta}(z)| = 0$ for all $z$ such that $g(\infty,z) < \delta$.  If $g(\infty,z) - \epsilon - \delta > 0,$ then $ \lim_{n \rightarrow \infty} |S_{n,\delta}(z)| = \infty$ by the left hand side of inequality  (\ref{eq:bounds}).  Since $\epsilon>0$ was arbitrary, we have $ \lim_{n \rightarrow \infty} |S_{n,\delta}(z)| = \infty$ for all $z$ such that $g(\infty,z) > \delta$.

\end{proof}


\section{Julia sets} \label{s:JuliaSets}

For each $n \in \mathbb{N}$ and $\delta > 0$, define the polynomial $P_{n,\delta}:\hat{\mathbb{C}} \rightarrow \hat{\mathbb{C}}$ by \begin{equation} \label{eq:Pdef}
P_{n,\delta}(z) = z S_{n,\delta}(z) = e^{n(\gamma - \delta)} z  \prod_{i=1}^n (\zeta_i^n - z)
\end{equation}
for $z \not = \infty$ and $P_{n,\delta}(\infty) = \infty$.  

\begin{proposition} \label{p:topological}
Assume $0 \in E$, and let $E^{\prime}$ be a compact subset of $E$ that contains $0$. Let $\Omega^{\prime}$ be a compact subset of $\textrm{int}(\Omega) \subset \hat{\mathbb{C}}$ that contains $\infty$.  Then there exists $D > 0$ such that for any $0 < \delta \leq D$, there exists a natural number $N(\delta)$ such that $n \geq N(\delta)$ implies $E^{\prime} \subset \mathcal{K}(P_{n,\delta})$ and $\Omega^{\prime} \subset \hat{\mathbb{C}} \setminus \mathcal{K}(P_{n,\delta})$
 \end{proposition}

\begin{proof}
Fix $r>0$ small enough that $B_r(0) \subset E^{\prime}$.  Fix a real number $m > 0$ such that $$m < \frac{r}{\sup \{|z|:z \in E^{\prime}\}}.$$  Fix $R > 0$ large enough that $\hat{\mathbb{C}} \setminus B_R(0) \subset \Omega^{\prime}$.  Fix a real number $M > 1$ such that $$M > \frac{R}{\inf \{|z| : z \in \Omega^{\prime}\}}.$$

We now show that there exists $D > 0$ such that for any $0 < \delta \leq D$, there exists $N(\delta) \in \mathbb{N}$ such that $n \geq N(\delta)$ implies $|S_{n,\delta}(z)| < m$ for all $z \in E^{\prime}$ and $|S_{n,\delta}(z)| > M$ for all $z \in \Omega^{\prime}$.
Recall that $g(\infty,z) = 0$ for $z \in E \cup \Gamma$ and $g(\infty,z) > 0$ for $z$ in the interior of $\Omega$.  Since $\Omega^{\prime}$ is a compact subset of the interior of $\Omega$, there exists $D> 0$ such that $g(\infty,z) > D$ for all $z \in \Omega^{\prime}$.  Let 
$$\alpha = \inf_{z \in \Omega^{\prime}} \{g(\infty,z)-D\}.$$  Notice $\alpha > 0$ by compactness of $\Omega^{\prime}$. Using $\epsilon = \alpha/2$, by equation (\ref{eq:bounds}), there exists $N_0 \in \mathbb{N}$ such that $n \geq N_0$ and $0 <\delta \leq D$ imply 
$$(e^{\alpha/2})^n \leq |S_{n,\delta}(z)|$$ for all $z \in \Omega^{\prime}$.   The sets 
$$\{z \in \Omega^{\prime}: (e^{\alpha/2})^k > M \textrm{ for all integers }k \geq n \}, \quad n \in \mathbb{N},$$
 form a countable open cover of $\Omega^{\prime}$; by compactness, this cover admits a finite subcover.  Hence there exists $N_1 \in \mathbb{N}$ such that $n \geq N_1$ and $0 < \delta \leq D$ imply  $|S_{n,\delta}(z)| > M$ for all $z \in \Omega^{\prime}$.  
Now for any fixed $0 < \delta \leq D$, using $\epsilon = \delta/2$ in equation (\ref{eq:bounds}), there exists $N_2 \in \mathbb{N}$ such that $|S_{n,\delta}(z)| < (e^{-\delta/2})^n$ for all $z \in E^{\prime}$ and $n \geq N_2$.  Pick $N_2$ to be large enough that $(e^{-\delta/2})^{N_2} < m$.  Set $N(\delta) = \max\{N_1,N_2\}$.  

Now fix $0 < \delta \leq D$  and let $n \geq N(\delta)$.  For $z \in E^{\prime}$, 
$$|P_{n,\delta}(z)| = |z| \cdot |S_{n,\delta}(z)| \leq \sup \{|z|:z \in E^{\prime}\} \cdot m < r.$$  Hence $z \in E^{\prime}$ implies $P_{n,\delta}(z) \in E^{\prime}$, and thus $E^{\prime} \subset \mathcal{K}(P_{n,\delta})$.  
For $z \in \Omega^{\prime}$, 
$$|P_{n,\delta}(z)| = |z| \cdot |S_{n,\delta}(z)| \geq \inf \{|z| : z \in \Omega^{\prime}\} \cdot M > R.$$  Hence $z \in \Omega^\prime$ implies $P_{n,\delta}(z) \in \Omega^{\prime}$.  For $z \in \Omega^{\prime}$, we also have $|P_{n,\delta}(z)| > |z| M$, with $M > 1$, so $\lim_{k \rightarrow \infty} |P^k_{n,\delta}(z)| = \infty$.  Hence $\Omega^\prime \subset \hat{\mathbb{C}} \setminus \mathcal{K}(P_{n,\delta})$.
\end{proof}

\begin{proposition} \label{p:metricprop}
Let $E$ be a finite union of disjoint smooth Jordan domains in $\mathbb{C}$ that contains $0$ and let $\epsilon > 0$.  Then there exists $D>0$ such that for any $0 < \delta \leq D$ there exists $N(\delta) \in \mathbb{N}$ such that $n \geq N$ implies 
$$d\left(E,\mathcal{K}(P_{n,\delta})\right) < \epsilon, \quad d\left(\partial E, \mathcal{J}(P_{n,\delta})\right) < \epsilon, \textrm{ and } d\left(\hat{\mathbb{C}} \setminus E, \hat{\mathbb{C}} \setminus \mathcal{K}(P_{n,\delta})\right) < \epsilon.$$
\end{proposition}

\begin{proof}[Proof of Theorem \ref{t:main}]
Fix $\epsilon > 0$.  For each Jordan curve $\Gamma_i$ of $\Gamma = \partial E$, pick a skinny open annulus $A_i \supset \Gamma_i$; let $A = \bigcup A_i$, let $\Omega^{\prime}$ be the unbounded connected component of $\hat{\mathbb{C}} \setminus A$, and let $E^{\prime} = \hat{\mathbb{C}} \setminus (\Omega^{\prime} \cup A)$.  We may pick the annuli $A_i$ to be skinny enough that $A \subset N_{\epsilon}(\Gamma)$ and for every $y \in \Gamma$, $B_{\epsilon}(y)$ has nonempty intersection with $E^{\prime}$ and with $\Omega^{\prime}$.  (This is possible because $\Gamma$ is compact.) We may further assume that the annuli are skinny enough that $0 \in E^{\prime}$.  Let $P$ be one of the polynomials $P_{n,\delta}$ from the statement of Proposition \ref{p:topological}.  

Suppose $x \in \mathcal{K}(P)$; then $x \in E$ or $x \in A\setminus E$.  If $x \in A\setminus E$, then $d(x,\Gamma) < \epsilon$ since $A \subset N_{\epsilon}(\Gamma)$, so $d(x,E) < 2\epsilon$.  Hence $x \in N_{2\epsilon}(E)$.  Now suppose $x \in E$; then either $x \in E^{\prime}$ or $x \in A$.  If $x \in A$, then $d(x,\mathcal{K}(P)) \leq d(x,E^{\prime}) < 2\epsilon$.  Hence $x \in N_{2\epsilon}(\mathcal{K}(P))$.  Thus $d(E,\mathcal{K}(P)) < 2\epsilon$.  A similar argument shows $d(\Omega, \hat{\mathbb{C}} \setminus \mathcal{K}(P)) < 2\epsilon$. 

Now suppose $x \in \Gamma$.  Then $x \in A$, and $B_{\epsilon}(x)$ contains both a point $u \in \Omega^{\prime} \subset \hat{\mathbb{C}} \setminus \mathcal{K}(P)$ and a point $v \in E^{\prime} \subset \mathcal{K}(P)$.  The straight line path from $u$ to $v$ is in $B_{\epsilon}(x)$ and must contain a point of $\mathcal{J}(P)$.  Hence $x \in N_{2\epsilon}(\mathcal{J}(P))$. Now suppose $x \in \mathcal{J}(P)$.  Then $x \in A$, and $A \subset N_{\epsilon}(\Gamma)$.  Thus $d(\Gamma, \mathcal{J}(P)) < 2 \epsilon$.  
\end{proof}

\begin{proof}[Proof of Theorem \ref{t:main}]
Without loss of generality, we assume $\partial E$ consists of smooth Jordan curves.  If $0 \in E$, we may take 
$P$ to be any of the polynomials $P_{n,\delta}$ associated to $E$ in Proposition \ref{p:metricprop}.  If $0 \not \in E$, let $\widetilde{E}$ be a translated copy of $E$ that contains $0$.  We may then take $P$ to be the conjugation by the translation of any of the polynomials $P_{n,\delta}$ associated to $\widetilde{E}$ in Proposition \ref{p:metricprop}.

\end{proof}

\bibliographystyle{plain}
\nocite{*}
\bibliography{FurtherShapesBib}

\end{document}